\documentclass[11pt]{amsart}
\usepackage{amsmath,amssymb,amscd,fullpage}

\theoremstyle{plain}

\newtheorem{thm}[equation]{Theorem}

\newtheorem{prop}[equation]{Proposition}
\newtheorem{lem}[equation]{Lemma}

\newtheorem{rmks}[equation]{Remarks}

\numberwithin{equation}{subsection}

\newcommand{\C}{\mathbb{C}}
\newcommand{\G}{\mathbb{G}}
\newcommand{\A}{\mathbb{A}}

\newcommand{\Z}{\mathbb{Z}}

\newcommand{\N}{\mathbb{N}}

\newcommand{\Tr}{\operatorname{Tr}}
\newcommand{\Fr}{\operatorname{Fr}}
\newcommand{\Gal}{\operatorname{Gal}}
\newcommand{\Ind}{\operatorname{Ind}}
\newcommand{\Norm}{\operatorname{Norm}}
\newcommand{\Sym}{\operatorname{Sym}}

\newcommand{\quo}{(F\backslash \A)}

\newcommand{\su}{SU_{2,1}}

\title{The Adjoint $L$-function of $SU_{2,1}$}
\author{Joseph Hundley}
\begin{document}
\maketitle

{\sl To the memory of my grandfather, Harold H. Hensold, Jr.}

In these notes we give a construction for a certain $L$-function attached to a globally generic automorphic representation of the quasi-split unitary group in $3$ variables associated to a quadratic extension $E/F$ of number fields.  
Recall that the finite Galois form of the $L$-group of this group is a semidirect product of $GL_3(\C)$ and $\Gal(E/F).$  The representation we consider has the property that when restricted to $GL_3(\C)$ it is the adjoint representation of this group.  For this reason, we refer to the associated $L$ function as the adjoint $L$ function.  In fact, ther are two representations of $GL_3(\C)\rtimes \Gal(E/F)$ with the above property-- related to one another by twistng by the unique nontrivial one dimensional representation of $\Gal(E/F).$  We pin down precisely which on we are talking about in section \ref{s:defOfRepr}
below.  Let us mention that a small modification of this construction gives the other.

The construction is a slight modification of that given in \cite{Ginzburg-AdjointSL3}.

\section{Notation}
Let $F$ be a global field, and $\A$ its ring of ad\`eles.  Let $E=F(\tau)$ be a quadratic extension, such that $\rho:=\tau^2\in F.$  Let $J=\left(\smallmatrix 1&&\\&1&\\&&1\endsmallmatrix\right).$  Abusing notation, we will also denote by $J$ the analogous matrix of any size with points in any ring (with unity).
The $F$ points of our special unitary group may be thought of as the set of $3\times 3$ matrices with determinant $1$ with entries in $E$ such that $g J\;{}^t\bar g=J.$  Here $\bar{\phantom{g}}$ denotes conjugation by the nontrivial element of $\Gal(E/F).$  Presently we shall also identify this group with a group of matrices having entries in $F.$  We denote this group by $SU_{2,1}.$  

We consider also the split exceptional group of type $G_2$ defined over $F,$ which we denote 
simply by $G_2.$  We recall a few facts about this group.  (See, \cite{F-H}, pp. 350-57.)
First, it may be realized as the identity component of the group of automorphisms of a seven dimensional vector space which preserve a general skew-symmetric trilinear form.  Second, this seven dimensional ``standard'' representation of $G_2$ is orthogonal:  the image also preserves a symmetric bilinear form.  We wish now to pin things down explicitly.  It will be convenient to realize $G_2$ as a subgroup of $SO_8.$  

Thus, we consider $SO_8=\{g\in GL_8: gJ\;^tg=J\}.$  Let $v_0=^t(0,0,0,1,-1,0,0,0).$  By $SO_7$ 
we mean the stabilizer of $v_0$ in $SO_8.$  
Let $V_0$ denote the orthogonal complement of $v_0,$ defined relative to $J.$  
To fix an embedding of $G_2,$ into $SO_7,$ we fix a trilinear form of $V_0$ in general 
position, namely:  
$$T:=e_7^*\wedge(e_4^*+e_5^*)\wedge e_2^*+e_1^*\wedge(e_4^*+e_5^*)\wedge e_8^* +e_6^*\wedge (e_4^*+e_5^*)\wedge e_3^* +2e_3^*\wedge e_2^* \wedge e_8^* -2e_6^*\wedge e_7^*\wedge e_1^*$$
(which is obtained from the form written down on p. 357 of \cite{F-H} via suitable identifications).
The identity component of the stabilizer of $T$ in $GL(V_0)$ is a group of type $G_2,$ defined and split over $F,$ and contained in $SO_7$ as defined above.  

Now, let $v_\rho ={}^t(0,0,1,0,0,\rho,0,0),$ and let $H_\rho$ denote the stabilizer of $v_\rho$ in $G_2.$

\begin{lem}
$H_\rho \cong SU_{2,1}.$
\end{lem}
\begin{rmks}
\begin{enumerate}
\item This is essentially the same embedding of $SU_{2,1}$ into $G_2$ described on p. 371 of \cite{BenArtzi-Ginzburg}.  
\item One may also obtain this embedding by making the following identifications between an $F$-basis of $E^3$ and one for the orthogonal complement of $\langle v_0,v_\rho\rangle$ in $F^8.$
$$\begin{array}{rclcrclcrcl}
(1,0,0)&\leftrightarrow& e_1 
& \quad &
(-\tau^{-1},0,0)&\leftrightarrow& e_2
&\quad& 
(0,-2\tau,0)&\leftrightarrow&e_3-\rho e_6\\
(0,-2,0)&\leftrightarrow& e_4+e_5 
&\quad& 
(0,0,2\tau) &\leftrightarrow& e_7 
&\qquad& (0,0,2)&\leftrightarrow &e_8
\end{array}$$
\end{enumerate}
\end{rmks}
\begin{proof}
On the one hand, we know from \cite{R-S}, pp.808-810 we know that the stabilizer of a vector in this
representation having  nonzero length (relative to $J$) is isomorphic to either $SL_3$ or $SU(Q)$ 
for a suitable $Q.$  On the other hand, $H_\rho$ is clearly contained in the group of automorphisms of
the six dimensional complement of $v_\rho$ in $V_0$ which preserve both the original symmetric bilinear form and the skew symmetric form obtained by plugging in $v_\rho$ as one of the arguments of $T.$  This latter group is isomorphic to $U_{2,1}$ (with an isomorphism being given by the identification of bases above).  The result follows.  
\end{proof}
To aid in visualizing these groups and checking various assertions below, we write down a general 
element of each of their Lie algebras.
\begin{eqnarray}
\label{e:G2LieAlg}
G_2: \begin{pmatrix}
T_1&a&c&d&d&e&f&0\\
g&T_2-T_1&b&-c&-c&d&0&-f&\\
h&l&2T_1-T_2&a&a&0&-d&-e\\
i&-h&g&0&0&-a&c&-d\\
i&-h&g&0&0&-a&c&-d\\
j&i&0&-g&-g&T_2-2T_1&-b&-c\\
k&0&-i&h&h&-l&T_1-T_2&-a\\
0&-k&-j&-i&-i&-h&-g&-T_1\end{pmatrix}\\
\label{e:SU21LieAlg}
SU_{2,1}:
\begin{pmatrix}
T_1&a&-\rho e&d&d&e&f&0\\
\rho a&T_1&-\rho d&\rho e&\rho e&d&0&-f\\
h&l&0&a&a&0&-d&-e\\
-\rho l&-h&\rho a&0&0&-a&-\rho e&-d\\
-\rho l&-h&\rho a&0&0&-a&-\rho e&-d\\
-\rho h&-\rho l&0&-\rho a&-\rho a&0&\rho d&\rho e\\
k&0&\rho l&h&h&-l&-T_1&-a\\
0&-k&\rho h&\rho l&\rho l&-h&-\rho a&-T_1\\
\end{pmatrix}
\end{eqnarray}

The set of upper triangular matrices in $G_2$ is a Borel subgroup $B_{G_2}$, and the 
set of diagonal matrices in $G_2$ is a maximal torus $T_{G_2}.$  We use this torus and Borel
to define notions of ``standard'' for parabolics and Levis.  We also fix a maximal 
compact subgroup $K=\prod_v K_v$ of $G_2(\A)$ such that 
$G_2(F_v)=B_{G_2}(F_v)K_v$ for all $v$ and $K_v=G_2(\frak o_v)$ for almost all finite $v.$  
(Here $\frak o_v$ denotes the ring of integers of $F_v.$)

For any matrix $A$ we let $_tA$ denote the ``other transpose'' $J\;^tAJ,,$
obtained by reflecting $A$ over the diagonal that runs from upper right to lower left.
Finally, if $H$ is any $F$ group, then $H\quo := H(F)\backslash H(\A).$

\subsection{The representations $r$}
Let us now describe explicitly the representation $r$ which appears in the Langlands $L$ function we will construct.  We first describe the $L$-group we consider, which is the finite Galois form of the $L$ group of $U_{2,1}(E/F).$  Let $\Fr$ denote the nontrivial element of $\Gal(E/F).$  Our $L$ group is $GL_3(\C) \rtimes \Gal(E/F),$ where the semidirect product structure is such that 
\begin{equation}\label{e:frobconj}
\Fr \cdot g \cdot \Fr = \;_tg^{-1}.  
\end{equation}
Now consider the $8$ dimensional complex vector space of $3\times 4$ traceless matrices, with an action of $GL_3(\C)$ by conjugation.  The definition 
\begin{equation}\label{e:frobact}
\Fr\cdot X = \; _tX\end{equation}
extends this to a well-defined action of $GL_3\rtimes \Gal(E/F).$  This is our representation $r.$

It is not difficult to check that there is only one other way to define the action of $\Fr$ which is 
compatible with \eqref{e:frobconj} and \eqref{e:frobact}, namely $\Fr\cdot X= -{}_tX.$  Now, 
it is part of the $L$-group formalism that the parameter of $\pi$ at an unramified place $v$ is in the identity component iff $\rho$ is a square in the completion of $F$ at $v.$  Hence, if we let $r'$ denote the representation corresponding to the action $\Fr\cdot C=-{}_tX,$ then 
$L^S(s,\pi,r')$ is the twist of $L^S(s,\pi,r)$ by the quadratic character corresponding to the 
extension $E/F.$  An integral for this $L$-function may be obtained from the one considered in this 
paper by inserting this character into the induction data for the Eisenstein series.
\label{s:defOfRepr}

\subsection{Eisenstein series}
We shall make use of the same Eisenstein series on $G_2\quo$ as in  \cite{Ginzburg-AdjointSL3}.
We recall the definition.  Let $P$ denote the standard maximal parabolic of $G_2$ such that the short simple root of $G_2$ is a root of the Levi factor of $P.$  Take $f$ a $K$-finite flat section of the 
fiber-bundle of representations $\Ind_{P(\A)}^{G_2(\A)}|\delta_P|^s$
(non-normalized induction).  
Thus, for each $s,$ $f(g,s)$ is a function $G_2(\A)\to \C$ such that $f(pg,s)=|\delta_P(p)|^s f(g,s)$
for all $p\in P(\A), g\in G_2(\A),$ and for $k \in K,$ the value of $f(k,s)$ is independent of $s.$ 
The associated Eisenstein series is defined by the formula
$$E(g,s)=\sum_{\gamma \in P(F)\backslash G_2(F)} f(\gamma g,s)$$
for $\Re(s)$ sufficiently large, and by meromorphic continuation elsewhere.

\section{Unfolding}
\begin{lem}
The space $P(F)\backslash G_2(F)/SU_{2,1}(F)$ has two elements, represented by the 
identity and (any representative in $G_2(F)$ for) the simple reflection in the Weyl group of $G_2$ associated to the long simple root,
which we denote $w_2.$ 
\end{lem}
\begin{proof}
This follows easily from our characterization of $SU_{2,1}$ as a stabilizer.  Indeed, $P(F)\backslash G_2(F)/SU_{2,1}(F)$ may be identified with the set of $P(F)$ orbts in the $G_2(F)$-orbit of $v_\rho.$  Write $v \in G_2(F)v_\rho$
as $^t(v_1,v_2,v_3)$ with $v_1,v_3\in F^2$ and $v_2\in F^4.$  Either $v_3=0$ or not.  This distinction clearly separates $P(F)$ orbits, and in particular separates the $P(F)$-orbit of the identity from that of $w_2.$  On the other hand, an element of the $G_2(F)$-orbit of $v_\rho$ is certainly in $V_0$ and of norm $2\rho.$  It is not hard to check that $P(F)$ permutes the set of such elements with $v_3=0$ and $v_3\ne 0$
each transitively. 
\end{proof}

Let $\varphi_\pi$ be a cusp form in the space of an irreducible automorphic cuspidal representations
$\pi$ of $SU_{2,1}(\A).$  Let $N$ denote the maximal unipotent subgroup of $SU_{2,1}$
$$\left\{\begin{pmatrix} 1&x&y\\&1&-\bar x\\ &&1\end{pmatrix}: x,y \in E, \Tr y+\Norm x =0\right\}.$$
Fix a nontrivial additive character $\psi$ of $\quo,$ and let $\psi_N$ denote the character of $N(\A)$ 
with coordinates as above to $\psi(\frac 12 \Tr x).$  (The $\frac 12$ is for convenience:  it cancels the $2$ that ariss when we take the trace of an element of $F.$)

We assume that the integral 
\begin{equation}\label{e:whittaker}
W_{\varphi_\pi}(g):=\int_{N\quo}\varphi_\pi(ng)\psi_N(n) \; dn\end{equation}
does not vanish identically.  (And hence, that $\pi$ is generic.)  We consider the integral 
$$I(\varphi_\pi, f,s):=\int_{SU_{2,1}\quo}\varphi_\pi(g)E(g,s).$$

\begin{thm} {\rm (The Unfolding)}  Let $N_2$ denote the two-dimensional unipotent 
subgroup of $SU_{2,1}$ corresponding to the coordinates $e$ and $f$ in \eqref{e:SU21LieAlg}.
Then for $\Re(s)$ sufficiently large, 
\begin{equation}\label{e:Unfolding}
I(\varphi_\pi, f,s) = \int_{N_2(\A)\backslash SU_{2,1}(\A)}
W_{\varphi_\pi}(g)f(w_2g,s) dg.
\end{equation}
\end{thm}
\begin{proof}
By the lemma, we find that $I(\varphi_\pi,f,s)$ is equal to 
$$\int_{(\su\cap P)(F)\backslash \su(\A)}
\varphi_\pi(g)f(g,s)dg+
\int_{\su\cap w_2Pw_2)(F)\backslash \su(\A)}
\varphi_\pi(g)f(w_2g,s) dg.$$
The first integral vanishes by the cuspidality of $\pi.$ 

The group $\su \cap w_2Pw_2$ consists of the one-dimensional $F$-split torus and the two-dimensional unipotent group $N_2.$  Incidentally, when written as elements of $GL_3(E),$ this unipotent group 
is 
$$\left\{\begin{pmatrix}
1&r\tau&t\tau +\frac{r^2\rho}2\\
&1&r\tau \\&&1\end{pmatrix}: r,t\in F
\right\}.$$
We now expand $\varphi_\pi$ along the subgroup of elements of the form 
$$\begin{pmatrix} 1&s&-\frac{s^2}2\\ &1&-s\\&&1\end{pmatrix}.$$
The constant term vanishes by cuspidality.  The remaining terms are permuted simply transitively by the action of the $F$-split torus.  The term corresponding to $1$ yields the integral \eqref{e:whittaker}.
\end{proof}

\section{Unramified computations}
We now consider the value of the local analogue of \eqref{e:Unfolding} at a place where all
data is unramified.  Thus, let $F$ be a nonarchimedean local field.  We denote the nonarchimedean 
valuation on $F$ by $v,$ and the cardinality of the residue field by $q.$
We keep the definitions of all 
the algebraic groups above.  However, we now allow the possibility that $\rho$ is a square in $F.$  In 
this case the group $\su$ defined by the equations above is isomorphic to $SL_3$ over $F.$  
We assume that $\rho$ and $2$ are both units in $F.$  In this section we encounter only the $F$-points
of algebraic groups, so we suppress the ``$(F).$'' 

Let $f$ be the spherical vector in the 
induced representation $\Ind_P^{G_2} |\delta_P|^s,$ and let $W$ denote the normalized spherical 
vector in the Whittaker model of an unramified local representation $\pi$ of $GL_3.$ The integral 
we consider is 
$$I(s,\pi) =\int_{N_2\backslash \su}W(g)f(w_2g,s)dg.$$
The main result of this section is the following:
\begin{prop}\label{p:Unramified}
For $\Re(s)$ sufficiently large, 
$$I(s,\pi) = 
\frac{L(3s-1,\pi,r)}{\zeta(3s)\zeta(6s-2)\zeta(3s-9)}.$$
Here, all zeta and $L$ functions are local.  Thus, if $q$ is the number of elements in the residue field of $F,$ then $\zeta(3s)=(1-q^{-3s})^{-1},$ etc.
\end{prop}
\begin{proof}
We begin with some computations which are valid regardless of whether or not $\rho$ is a square in $F.$  The one-dimensional subgroup of $\su$ corresponding to the variable $d$ in \eqref{e:SU21LieAlg}
maps isomorphically onto the quotient $N_2\backslash N.$  An element of this group may also be 
expressed as $x_{\alpha_2}(\rho u) x_{2\alpha_1+\alpha_2}(-u),$
 where $x_{\alpha_2}$ and $x_{2\alpha_1+\alpha_2}$ are maps of $\G_a$ onto the one parameter unipotent subgroups oof $G_2$ corresponding to the indicated roots.  These subgroups correspond to the variables $b$ and $d$ in 
 \eqref{e:G2LieAlg}.  Using the Iwasawa decomposition, we express the integral over $N_2\backslash \su$ as integrals over the maximal compact $K,$ the torus $T,$ and this one-dimensional subgroup.  
 Since $W$ and $f$ are spherical, and the volume of $K$ is $1,$ we may erase the integrall over $K.$  Also $w_2x_{2\alpha_1+\alpha_2}(u) \in P.$  Hence, we find that 
 $$I(s,\pi)=\int_T\int_F f(w_2x_{\alpha_2}(\rho u)t,s)
 \psi_u du W(t) \delta^{-1}_B(t) dt.$$
 Here $\delta_B$ denotes the modular quasicharacter of the Borel subgroup of $\su.$  
 Now, an element of $T$ may be visualized as an element of $SL_3(F(\sqrt{\rho}))$ of the form 
 $$\begin{pmatrix} a+b\sqrt{\rho}&&\\&\frac{a-b\sqrt{\rho}}{a+b\sqrt{\rho}}&\\&&\frac1{a-b\sqrt{\rho}}\end{pmatrix}.$$  
 The corresponding $8\times 8$ matrix is 
 $$\begin{pmatrix}
 a&-b&&&&&&\\
 -b\rho&a&&&&&&\\
 &&\frac{a^2}N&-\frac{ab}N&-\frac{ab}N&-\frac{b^2}N&&\\
 &&-\frac{ab\rho}N&\frac{a^2}N&\frac{b^2\rho}N&\frac{ab}N&&\\
 &&-\frac{ab\rho}N&\frac{b^2\rho}N&\frac{a^2}N&\frac{ab}N&&\\
 &&-\frac{b^2\rho^2}N&\frac{ab\rho}N&\frac{ab\rho}{N}&\frac{a^2}N&&\\
 &&&&&&\frac aN&\frac bN\\
 &&&&&&\frac{b\rho}N&\frac aN\\
 \end{pmatrix}, \qquad \text{ where }N:=a^2-b\rho^2.$$
We now write the Iwasawa decomposition for this as an element of $G_2.$  First, suppose that 
$|b\rho|\le |a|.$  Then the decomposition is 
$$\left(\begin{smallmatrix}
1&-\frac ba&&&&&&\\
&1&&&&&&\\
&&1&-\frac ba&-\frac ba&-\frac{b^2}{a^2}&&\\
&&&1&&\frac ba&&\\
&&&&1&\frac ba&&\\
&&&&&1&&\\
&&&&&&1&\frac ba\\
&&&&&&&1\\
\end{smallmatrix}\right)\left(\begin{smallmatrix}
\frac Na&&&&&&&\\
&a&&&&&&\\
&&\frac N{a^2}&&&&&\\
&&&1&&&&\\
&&&&1&&&\\
&&&&&\frac{a^2}N&&\\
&&&&&&\frac 1a&\\
&&&&&&&\frac aN\\
\end{smallmatrix}\right)\left(\begin{smallmatrix}
1&&&&&&&\\
\frac{b\rho}a&1&&&&&&\\
&&1&&&&&\\
&&\frac{b\rho}a&1&&&&\\
&&\frac{b\rho}a&&1&&&\\
&&-\frac{b^2\rho^2}{a^2}&-\frac{b\rho}a&-\frac{b\rho}a&1&&\\
&&&&&&1&\\
&&&&&&-\frac{b\rho}a&1\\
\end{smallmatrix}\right).$$
If, $|b\rho|>|a|,$ it is
$$\left(\begin{smallmatrix}
&&&&&&&\\
&&&&&&&\\
&&&&&&&\\
&&&&&&&\\
&&&&&&&\\
&&&&&&&\\
&&&&&&&\\
&&&&&&&\\
\end{smallmatrix}\right)\left(\begin{smallmatrix}
\frac N{b\rho}&&&&&&&\\
&b\rho&&&&&&\\
&&\frac{N}{b^2\rho^2}&&&&&\\
&&&1&&&&\\
&&&&1&&&\\
&&&&&\frac{b^2\rho^2}N&&\\
&&&&&&\frac1{b\rho}&\\
&&&&&&&\frac{b\rho}N\\
\end{smallmatrix}\right)\left(\begin{smallmatrix}
&&&&&&&\\
&&&&&&&\\
&&&&&&&\\
&&&&&&&\\
&&&&&&&\\
&&&&&&&\\
&&&&&&&\\
&&&&&&&\\
\end{smallmatrix}\right).$$
Let us denote the three factors by $u',t'$ and $k',$ respectively.  Then $u'$ has the 
property that 
$w_2u'w_2^{-1}$ and 
$w_2[x_{\alpha_2}(u),u']w_2^{-1}$ (where $[\,,\,]$ denotes the commutator) are both in $P.$ 
Thus
$f(w_2x_{\alpha_2}(u)t,s)=f(w_2x_{\alpha_2}(u)t',s).$
We have
$$I(s,\pi)=\int_T\left(\int_F f(w_2x_{\alpha_2}(u),s)\psi(\alpha_2(t')u) du\right)
K(t) \delta_B^{-\frac12}(t) \delta_P^s(w_2t'w_2)|\alpha_2(t')|dt,$$
where $t'$ is as above, and $K(t):=W(t)\delta_B(t)^{-\frac 12}.$  We find that 
$$\delta_B^{-\frac12}(t)=|N|^{-1},\qquad
\delta_P(w_2t'w_2)=\frac{|N|^2}{\max(|a|,|b|)^3},
\qquad |\alpha_2(t')|=\frac{\max(|a|,|b|)^3}{|N|}.$$
\begin{lem}
$$\int_F f(w_2x_{\alpha_2}(u),s)\psi(cu) du=(1-q^{-3s})\frac{(1-q^{(-3s+1)(v(c)+1)})}{(1-q^{-3s+1})}.$$
\end{lem}
\begin{proof}
There is an embedding $j$ of $SL_2$ into $G_2$ such that $j\left(\begin{smallmatrix} &1\\-1&\end{smallmatrix}
\right)=w_2$ and $j\left(\begin{smallmatrix} 1&u\\&1\end{smallmatrix}\right)=x_{\alpha_2}(u).$
The lemma is a well-known computation from $SL_2$ applied to this copy oof $SL_2.$  One has only to check that 
$f\left(j\left(\begin{smallmatrix}t&\\&t^{-1}\end{smallmatrix}\right),s\right)=t^{-3s}.$
\end{proof}
  Let $x=q^{-3s+1}.$  Then the above reads
$$(1-q^{-1}x)\frac{(1-x^{v(c)+1})}{(1-x)}.$$
To complete the argument, we must consider the two cases ($\su$ splits over $F$ or does not)
separately.

\subsubsection{Split Case}  In this case put $t_1=a+b\sqrt{\rho}$ and $t_2=a-b\sqrt{\rho}.$  
Then $t_1$ and $t_2$ are just two independent variables ranging over $F^\times.$  The quantity called
``$N$'' above us equal to $t_1t_2.$  Since $\rho$ and $2$ are units, $\max(|a|,|b|)=\max(|t_1|,|t_2|).$  
Let $\beta_1,\beta_2$ denote the simple roots of $SL_3.$  We get
$$\delta_P(w_2t'w_2)=\frac{|t_1t_2|^2}{\max(|t_1|,|t_2|)^3}
=\min(|t_1^{-1}t_2^2|,|t_1^2t_2^{-1}|)=\min(|\beta_1(t)|,|\beta_2(t)|),$$
$$|\alpha_2(t')|=\frac{\max(|t_1|,|t_2|)^3}{|t_1t_2|}=\max(|\beta_1(t)|,|\beta_2(t)|).$$
Now define two integer-valued variables, depending on $t$ by $m_i=v(\beta_i(t)),\; i=1,2.$
As $t$ ranges over the torus of $SL_3,$ the pair $(m_1,m_2)$ ranges over
$$\{(m_1,m_2)\in \Z^2: m_1-m_2 \text{ is divisible by }3\}.$$
Every part of our local integral can now be expressed in terms of $m_1$ and $m_2.$  First, we consider the function $K(t).$  This is evaluated using the Casselman-Shalika formula \cite{C-S}.  
It is equal to zero unless $m_1$ and $m_2$ are both non-negative.  If $m_1$ and $m_2$ are both non-negative, then the pair corresponds to a dominant weight for the group $PGL_3(\C).$  Let 
$\Gamma_{m_1.m_2}$
denote the corresponding irreducible finite dimensional representation, which we may also regard as a representation of $GL_3(\C).$  Then we have
$$K(t)=\Tr \Gamma_{m_1,m_2}(\tilde t_\pi),$$ where $\tilde t_\pi$ is the conjugacy class of $GL_2(\C)$ associated to the local representation $\pi.$  Also 
$\delta_P(w_2t'w_2)=q^{-\max(m_1,m_2)},|\alpha_2(t')|=q^{-\min(m_1,m_2)},$ and $\delta_B^{-\frac12}(t)=q^{-m_1-m_2}.$  Thus, we consider,
$$(1-q^{-1}x)\sum_{m_1,m_2}\frac{1-x^{\min(m_1,m_2)+1}}{1-x}x^{\max(m_1,m_2)}\Tr \Gamma_{m_1,m_2}(\tilde t_\pi),$$
where the sum is over $m_1,m_2$ both non-negative, such that $m_1-m_2$ is divisible by $3.$

We now make use of the relationship between local Langlands $L$-functions and the Poincar\'e series of 
certain graded algebras.  We first review some definitions.  Fix $N\in \N,$ and 
$$A=\bigoplus_{i_1,\dots, i_N\in\N}A_{i_1,\dots, i_N}$$
a graded algebra over a field $k.$  The Poincar\'e series of $A$ is a power series in $N$ indeterminates
$$\sum_{i_1,\dots, i_N=0}^\infty \dim(A_{i_1,\dots,i_N})T_1^{i_1} \dots T_N^{i_N}.$$
The graded algebra which is relevant for consideration of Langlands $L$-functions is described as follows.  Let $^LG$ be a semisimple complex Lie group, and $(r,V)$ a finite-dimensional representation.  
Inside the symmetric algebra $\Sym^*(V)$ we consider the subalgebra $\Sym^*(V)^{^LU}$ of $^LU$-invariants.  
This subalgebra contains the highest weight vectors of each of the irreducible components 
of $Sym^*(V)$ 
and is graded by the semigroup of dominant weights of $^LG$ as well as by degree.  

Let us use a slightly different notation from that above.  We use $X$ for the indterminate associated 
to the grading by degree, and $T_1, \dots, T_N$ for the grading by weight.  Let $\pi$ be an 
unramified representation of $G(F)$ where $F$ is a non-archimedean local field  and $G$ is an
 algebraic group such that $^LG$ is the $L$-group.  Let $\tilde t_\pi$ be the semisimple conjugacy class in $^LG$ corresponding to $\pi.$  Then it follows from the definitions that the local Langlands $L$-function $L(s,\pi,r)$ may be obtained from the Poincar\'e series of $\Sym^*(V)^{^LU}$ by substitution $q^{-s}$ for $X$ and $\Tr\Gamma_{k_1\varpi_1+\dots +k_N\varpi_N}(\tilde t_\pi)$
 for $T_1^{k_1}\dots T_N^{k_N}.$  Here $\Gamma_{k_1\varpi_1+\dots +k_N\varpi_N}$ denotes the 
 irreducible finite dimensional representation of $^LG$ with highest weight
  ${k_1\varpi_1+\dots +k_N\varpi_N}.$
 
 In cases when $^LG$ is reductive but not semisimple, this discussion must be adapted, as the choice of maximal unipotent $^LU$ does not by itself pin down a basis for the wieght lattice which may be used to define the grading.  In the case at hand, one needs only to observe that, since the adjoint representation of $GL_3(\C)$ factors through the projection to $PGL_3(\C)$ (which is semisimple)
 each of the representations 
 appearing in the decomposition of the symmetric algebra must as well. 
 Alternatively, one may define the grading using the weights of the derived group.
 
 The following fact is well-known:
 
 \begin{lem}
 Let $\varpi_1,\varpi_2$ denote the fundamental weights of $PGL_3(\C).$  Let $T_1,T_2$ and $X$ be 
  indeterminates associated, as above, to the grading on $\Sym^*(V)^{^LU},$ where $V$ is the eight-dimensional adjoint representation.  Then the Poincar\'e series of $Sym^*(V)^{LU}$ may be summed, 
  yielding the rational function:
  $$\frac{1-T_1^3T_2^3X^6}{(1-T_1T_2X)(1-T_1T_2X^2)(1-T_1^3X^3)(1-T_2^3X^3)(1-X^2)(1-X^3)}.$$
 \end{lem}
 To complete the proof of proposition \ref{p:Unramified} in the split case, we just need to show that
 $$\sum_{{m_1,m_2=0}\atop{3|(m_1-m_2)}}\frac{1-X^{\min(m_1,m_2)+1}}{1-X}X^{\max(m_1,m_2)}T_1^{m_1}T_2^{m_2}=
 \frac{1-T_1^3T_2^3X^6}{(1-T_1T_2X)(1-T_1T_2X^2)(1-T_1^3X^3)(1-T_2^3X^3)},$$
  which is a straightforward computation.
 
\subsubsection{Non-Split Case}
Suppose now that $\rho$ is not a square in the local firld $F.$  Then it is part of the $L$-group
formalism that the semisimple conjugacy class $\tilde t_\pi$ in $^LG=GL_3(\C)\rtimes \Gal(E/F)$
associated to $\pi$ is in the coset corresponding to the nontrivial element of $\Gal(E/F),$ which we denote $\Fr.$  Each such conjugacy class contains an element of the form 
$$\left(\begin{pmatrix} \mu&&\\&\pm1&\\&&\mu^{-1}\end{pmatrix}, \Fr\right).$$
Adjusting by an element of the center, we may assume the sign in the middle is plus.

Referring to section \ref{s:defOfRepr}, we see that our eight-dimensional representation decomposes into 
a $5$ dimensional $+1$ eigenspace for $\Fr$ on which $\left(\begin{smallmatrix}\mu&&\\&1&\\
&&\mu^{-1}\end{smallmatrix}\right)$ acts with eigenvalues $\mu^2,\mu,1,\mu^{-1},\mu^{-2},$ and 
a $3$-dimensional $-1$ eigenspace for $\Fr,$ on which 
a $5$ dimensional $+1$ eigenspace for $\Fr$ on which $\left(\begin{smallmatrix}\mu&&\\&1&\\
&&\mu^{-1}\end{smallmatrix}\right)$ acts with eigenvalues $\mu,1,\mu^{-1}.$  Hence, the local $L$-function 
is 
$$\frac 1{(1-\mu^2x)(1-\mu^2x^2)(1-x^2)(1-\mu^{-2}x)(1-\mu^{-2}x^2)},$$
where $x=q^{-3s+1}$ as before.  
This may also be written as 
$$\frac1{(1-x^2)}
\sum_{k_1,k_2=0}^\infty 
\Tr(\Gamma_{k_1}\otimes\Gamma_{k_2})\left(\begin{smallmatrix}\mu^2&\\&\mu^{-2}\end{smallmatrix}\right)x^{k_1+2k_2}.$$
Turning to the local integral, we find that in this case we have $m_1=m_2.$  Let us therefore
denote this quantity simply ``$m.$''
\begin{lem}
With the notation as above, we have 
$$K(t)=\Tr\Gamma_m\left(\begin{smallmatrix}\mu^2&\\&\mu^{-2}\end{smallmatrix}\right).$$
\end{lem}
\begin{proof}
This can be verified either by direct computation or by a close reading of \cite{C-S}.  In either
method it is necessary first to identify the precise unremified character of the torus of $\su$ corresponding to 
$\left(\left(\begin{smallmatrix} \mu&&\\&\pm1&\\&&\mu^{-1}\end{smallmatrix}\right), \Fr\right).$
For the convenience of the reader, we record that it is the map sending the torus element with coordinates $a$ and $b$ as above to $\mu^{v(a^2-b^2\rho)}.$  (Here $v$ is again the discrete 
valuation on the field $F.$)

This case of the proposition now follows from the identity
$$\sum_{k_1,k_2=0}^\infty \sum_{i=0}^{\min(k_1,k_2)}X^{k_1+2k_2}T^{k_1+k_2-2i}
=\frac1{(1-X^3)(1-TX)(1-TX^2)}
=\frac 1{1-X^3}\sum_{m=0}^\infty X^m\frac{1-X^{m+1}}{1-X}T^M,$$
which is straightforward to verify
\end{proof}
This also completes the proof of Proposition \ref{p:Unramified}
\end{proof}

\end{document}